\documentclass[12pt]{amsart}

\usepackage{cancel}
\usepackage{latexsym, amssymb, amsmath,esint}
\usepackage{soul}
\usepackage{amsfonts, graphicx}
\usepackage{graphicx,color}
\newcommand{\be}{\begin{equation}}
\newcommand{\ee}{\end{equation}}
\newcommand{\beq}{\begin{eqnarray}}
\newcommand{\eeq}{\end{eqnarray}}

\newtheorem{thm}{Theorem}[section]

\newtheorem{lma}{Lemma}[section]
\newtheorem{prop}{Proposition}[section]
\newtheorem{cor}{Corollary}[section]

\newtheorem{que}{Question}[section]
\theoremstyle{remark}
\newtheorem{rem}{Remark}[section]
\numberwithin{equation}{section}

\def\be{\begin{equation}}
\def\ee{\end{equation}}
\def\bee{\begin{equation*}}
\def\eee{\end{equation*}}

\def\lf{\left}
\def\ri{\right}

\def\Ric{\text{\rm Ric}}

\def\wt{\widetilde}

\def\p{\partial}

\def\heat{\lf(\frac{\p}{\p t}-\Delta\ri)}
\def\tr{\operatorname{tr}}

\def\e{\varepsilon}

\def\a{{\alpha}}
\def\b{{\beta}}

\begin{document}

\title[]
{Rigidity of contracting map using harmonic map heat flow}

 \author{Man-Chun Lee}
\address[Man-Chun Lee]{Department of Mathematics, The Chinese University of Hong Kong, Shatin, Hong Kong, China}
\email{mclee@math.cuhk.edu.hk}

\author{Jingbo Wan}
\address[Jingbo Wan]{Department of Mathematics, Columbia University, New York, NY, 10027}
 \email{jingbowan@math.columbia.edu}

\date{\today}

\begin{abstract} 
Motivated by a question of Tsai-Tsui-Wang, we consider the rigidity of map from manifolds with positive Ricci curvature to manifolds with positive sectional curvature. We show that if the Ricci curvature of the domain dominates that of the target, then distance non-increasing maps must be either Riemannian submersion or isometry. The rigidity result also holds on a wider class of manifolds with positive curvature and weaker contracting property on the map. This is based on a new long-time existence of harmonic map heat flow.
\end{abstract}

\maketitle

\markboth{Man-Chun Lee, Jingbo Wan}{}
\section{introduction}

Let $M$ and $N$ be two compact Riemannian manifolds. There has been great interest in finding a local geometric condition which determine the homotopy class of map from $M$ to $N$. When $M$ and $N$ are both sphere (of possibly different dimensions), $1$-dilatation map or equivalently distance decreasing map, has been considered and studied extensively, \cite{DeturckGluckStorm,Gromov1978,Hsu1972,Lawson1986}. It is not difficult to see that any distance decreasing map between spheres are homotopically trivial. In \cite{Gromov1996}, Gromov proposed to consider maps which are area-decreasing. It was conjectured that any area-decreasing map from $\mathbb{S}^n$ to $\mathbb{S}^m$ must be homotopically trivial.

In \cite{TsuiWang2003}, Tsui-Wang confirmed the conjecture affirmatively. Their method is based on the graphical mean curvature flow. By embedding $M$ into $M\times N$ using $\mathrm{id}\times f$, they view the initial map as a graph. Their novel idea is to deform the graph using the mean curvature flow which is the negative gradient flow of the volume functional. By showing that the mean curvature flow remains graphical and converges to a constant graph smoothly, the  homotopically triviality of initial map follows. The main difficulties lies on establishing the long-time existence and convergence. This boils down to show that area-decreasing property is preserved along the flow. Efforts have been made in generalizing the result of Tsui-Wang to more general curvature conditions in both the ambient spaces and target manifolds. For instances, see \cite{AssimosSavasSmoczyk2023,LeeLee2011,SavasSmoczyk2014}. The mentioned works are all based on the studying the symmetric $2$-tensor introduced by Tsui-Wang \cite{TsuiWang2003} whose positivity is equivalent to the area decreasing conditions. Very recently, this stream of technique has been improved significantly by Tsai-Tsui-Wang \cite{TsaiTsuiWang2023} in which now includes maps between complex projective spaces, and maps from spheres to complex projective spaces, among others. We summarize their result as follows. 
\begin{thm}[Tsui-Wang \cite{TsuiWang2003}, Tsai-Tsui-Wang \cite{TsaiTsuiWang2023}]\label{TTW} Let $(\Sigma_1,g_1)$, $(\Sigma_2,g_1)$ be two compact, Riemannian manifolds. Let $n=\dim \Sigma_1$, $m=\dim \Sigma_2$. Let $f:\Sigma_1\rightarrow\Sigma_2$ be any area-decreasing map, then $f$ must be homotopically trivial if one of the following holds:
\begin{enumerate}
    \item[(a)] $n\geq 2$, sectional curvatures of $(\Sigma_1,g_1)$ and $(\Sigma_2,g_1)$ are constant and satisfy
    \[\mathrm{sec}(\Sigma_1)\geq |\mathrm{sec}(\Sigma_2)| \quad,\quad \mathrm{sec}(\Sigma_1)+\mathrm{sec}(\Sigma_2)>0\]
    \item[(b)] $n\geq m\geq 2$,  sectional curvatures of $(\Sigma_1,g_1)$ and $(\Sigma_2,g_1)$ satisfy
    \[\mathrm{sec}(\Sigma_1)\geq 1 \quad,\quad \mathrm{sec}(\Sigma_2)< \frac{2n-m-1}{m-1}\]
    \item[(c)] $n\geq m\geq 2$, Ricci and sectional curvatures of $(\Sigma_1,g_1)$, $(\Sigma_2,g_1)$ satisfy
    \[\frac{\Ric_1}{g_1}\geq \frac{\Ric_2}{g_2} \quad,\quad   \mathrm{sec}(\Sigma_1)+\mathrm{sec}(\Sigma_2)> 0 .\]
\end{enumerate}
\end{thm}

On the other hand in scalar curvature geometry, a Theorem of Llarull \cite{Llarull1998} states that area non-increasing map from a compact spin manifold $M^n$ with $\mathrm{scal}(M)\geq n(n-1)$, to $\mathbb{S}^n$ can either have zero degree or is an isometry. Motivated the above mentioned works, it was asked by Tsai-Tsui-Wang \cite{TsaiTsuiWang2023} whether an area non-increasing map between complex projective spaces is an isometry. Building upon the ideas in \cite{TsuiWang2003,TsaiTsuiWang2023}, the authors and Tam \cite{LeeTamWan2023} develop a strong maximum principle argument to study area non-increasing maps and thus answers the question of Tsai-Tsui-Wang affirmatively. Furthermore, the result of \cite{TsaiTsuiWang2023} had been extended to curvature which is mildly weaker than sectional curvature. We refer interested readers to \cite[Theorem 1.1]{LeeTamWan2023}.

In this work, we are interested in the homotopy and rigidity of maps in Ricci geometry. Our main result is the following. 
\begin{thm}\label{intro-main}
Let $(M^n,g)$ and $(N^m,h)$ be two compact manifolds with $\ell=\min\{m,n\}\geq 2$.  Suppose their curvatures satisfy
\begin{enumerate}
\item[(a)] $\mathrm{sec}(h)> 0$;
\item[(b)] there exists $\phi\in C^\infty(M)$ such that for all unit vector $u\in TM$ and $v\in TN$, 
$$\Ric^\phi_g(u,u)\geq \Ric_h(v,v)$$
where $\Ric^\phi=\Ric+\nabla^2 \phi$.
\end{enumerate}
If $f:M\to N$ is a smooth map which is not  is homotopically trivial so that $g-f^*h$ is $2$-nonnegative, then we must have $m\leq n$ and 
\begin{enumerate}
    \item[(I)] if $n>m$, then $f$ is a Riemanian submersion;
    \item[(II)] if $m=n$, then $f$ is a local isometry and $\phi$ is a constant on each connected component of $M$. Moreover, $M,N$ are Einstein manifolds with positive sectional curvature.
\end{enumerate}
\end{thm}

\begin{rem}
We remark here that since condition (b) holds for arbitrary points and unit vector, it is equivalent to say that the minimum of weighted  Ricci curvature lower bound of $g$ is larger than or equal to the maximum of the Ricci curvature of $h$.
\end{rem}

In particular, the Theorem applies when $M$ and $N$ are either $\mathbb{CP}^n$, $\mathbb{S}^n$ or more generally Einstein manifolds with positive Einstein constant, although they had been taken care by method in mean curvature flow. In the case of $M=N=$ $\mathbb{CP}^n$ or $\mathbb{S}^n$, we know $f$ must be an isometry because the domain is connected and target is simply connected. 

Our method is different from the approach taken in the earlier works. This is inspired by the recent work of the first named author and Tam \cite{LeeTam2022}. Instead of using mean curvature flow, we consider deformation using the harmonic map heat flow which is routine enough to treat the $\infty$-Bakry-Emery Ricci curvature (and hence the $N$-Bakry-Emery Ricci curvature). Notice that there exist many interesting examples of manifolds with $N$-Bakry-Emery Ricci lower bound, for instances, see \cite{WeiWylie2009}.  The harmonic map heat flow is introduced by Eells-Sampson \cite{EellsSampson1964} which aims to deform a given map to a harmonic one. If the target manifold $N$ is not non-positively curved, the flow will in general form finite time singularity. In case the Ricci curvature of $M$ dominate that of $N$, we are able to rule out the finite time singularity under $2$-nonnegativity condition of $g-f^*h$. This condition is stronger than the area non-increasing but weaker than the distance non-increasing. More precisely, if we take $\lambda_i$ to be the singular values of $df$ with respect to $g,h$, then $2$-nonnegativity condition of $g-f^*h$ is equivalent to say that $\lambda_i^2+\lambda_j^2\leq 2$ for all distinct $i,j$. This contractive property is in between distance non-increasing $||df||_\infty\leq 1$ and area non-increasing $||\Lambda^2 df||_\infty\leq 1$.  More concretely, we show that the $2$-nonnegativity condition will be preserved under the harmonic map heat flow in a similar spirit of \cite{TsuiWang2003} which avoid finite time singularity. In \cite{StepanovTsyganok2017}, a similar condition is considered in term of vanishing theorem for harmonic maps. In contrast with the mean curvature flow approach, the harmonic map heat flow is considerably more linear and straight forward. However, it seems out of reach to handle the area non-increasing maps as it is more sensitive to $\sigma_1$ operator rather than $\sigma_n$ operator.




On the other hand, the harmonic map heat flow seems to be more flexible in requiring the sectional curvature of $M$ even in the content of detecting homotopically triviality.  Along this direction, we have the following which might be of independent interest. 
\begin{thm}\label{intro-thm2} 
    Let $(M^n,g)$ and $(N^m,h)$ be two compact manifolds with $\ell=\min\{m,n\}\geq 2$.  Suppose their curvatures satisfy
\begin{enumerate}
\item[(i)] $\mathrm{sec}(h)\geq  0$;
\item[(ii)] there exists $\phi\in C^\infty(M)$ such that for all unit vector $u\in TM$ and $v\in TN$, 
$$\Ric^\phi_g(u,u)\geq \Ric_h(v,v).$$
\end{enumerate}
Let $f:M\to N$ be a smooth map so that $g-f^*h$ is $2$-positive, then $f$ is homotopically trivial if one of the following holds:
\begin{enumerate}
\item[(a)] $\Ric_g^\phi$ is positive somewhere;
\item[(b)] $\phi$ is a constant on each connected component of $M$ and $\mathrm{scal}(g)$ is positive somewhere.
\end{enumerate}
\end{thm}

As a simple consequence, in the case when $\phi$ is constant, then we have the following dichotomy: either $f$ is homotopically trivial; or $M$ and $N$ are both Ricci flat. This is because if $f$ is homotopically non-trivial where $g-f^*h$ is $2$-positive, then $\Ric_g\equiv 0$ thanks to $\Ric_g\geq \Ric_h\geq 0$ and thus $\Ric_h\equiv 0$. Given the nature of the method, we ask if the rigidity holds even on non-smooth metric space. 

\begin{que}
Let $(X^n,d_X,\mu_X)$ be a non-collapsed $\mathrm{RCD}(n-1,n)$ for some $n\in [1,+\infty)$. Suppose $f:(X,d_X)\to\mathbb{S}^n$ is a distance decreasing map to the standard sphere, is  $f$ homotopy trivial? 
\end{que}

{\it Acknowledgement:} The first named author was partially supported by Hong Kong RGC grant (Early Career Scheme) of Hong Kong No. 24304222, a NSFC grant and a direct grant of CUHK. The second named author would like to thank Professor Mu-Tao Wang for his encouragement and valuable discussions.

\section{Monotonicity along harmonic heat flow}\label{Sec: estimate}

In this section, we will prove some a-priori estimates of the (weighted) harmonic map heat flow. We start with recalling the definition of harmonic map heat flow. Suppose $f:M^n\to N^m$ is a smooth map, then $df$ is a section of $T^*M\otimes f^{-1}TN$ where $f^{-1}TN$ is the pull-back bundle induced by $f$ on $M$. We let $\nabla$ be the connection induced by metric $g$ on $M$ and $h$ on $N$. In this way, $\nabla df$ is called the second fundamental form which is a section in $T^*M\otimes T^*M\otimes f^{-1}TN$. The trace of $\tr_g \nabla df$ with respect to $g$ is called the tension field which is a vector field along $f(M)$. We will use $\Delta f$ to denote it\footnote{it is also sometimes denoted by $\tau(f)$.}. Given $(M,g)$ and $(N,h)$,  we consider the $\phi$-weighted harmonic map heat flow with initial data $f$ which is the evolution equation satisfying
\begin{equation}\label{eqn:HMHF}
    \left\{
    \begin{array}{ll}
        \partial_t F &=\Delta F- \langle dF, d\phi\rangle \\
         F(0)&=f 
    \end{array}
    \right.
\end{equation}
where $F:M\times [0,T]\to N$ is a one parameter family of smooth map from $M$ to $N$ and $\phi\in C^\infty(M)$. We will use $F_t$ and $F_i$ to denote $F_*(\partial_t)$ and $F_*(\partial_i)$ respectively.  We will call this a solution to the $\phi$-harmonic map heat flow.

We start with deriving an evolution equation of the eigenvalue of the pull-back metric $F^*h$. In what follows, we will use $\tilde R$ to denote the curvature of $h$ at $F(x)$. All norms and connections are with respect to the fixed metrics $g$ on $M$ and $h$ on $N$. To ease the notations, we will also denote $\ell=\min\{m,n\}$. 
\bigskip 

\begin{lma}\label{lma:evo-F}
Along the $\phi$-harmonic map heat flow between $(M^n,g) $ and $(N^m,h)$, the tensor $\a=g-F^*h$ satisfies 
    \begin{equation}\label{heatalpha}
\heat  \a_{ij}=(R^\phi)_i^l H_{lj} +(R^\phi)_j^l H_{il} -2g^{kl} (F^*\tilde R)_{kijl}+2g^{kl}F^\a_{ik} F_{jl}^\b h_{\a\b}-g^{kl}\phi_k \nabla_l \a_{ij}.
\end{equation}
where $H=F^*h$. In particular, 
\begin{equation} \label{heatgradF}
    \heat |d F|^2=-2|\nabla dF|^2 -2(R^\phi)^{ij}F_i^\a F_j^\b h_{\a\b} +2g^{kl}g^{ij} (F^*\mathrm{Rm}_h)_{iklj}+g^{kl}\phi_k \nabla_l |dF|^2.
\end{equation}
Here $\Delta$ denotes the Laplacian with respect to $g$ and $(R^\phi)_{ij}$ are local components of $\Ric^\phi=\Ric+\nabla^2\phi$, i.e. $(R^\phi)_{ij}=R_{ij}+\nabla_i\nabla_j\phi$, where we raise or lower its indices via metric components of $g$.
\end{lma}
\begin{proof}  
Let $H=F^*(h)$ so that $H_{ij}(x)=F^\a_i F^\b_j h_{\a\b}(F(x))$ for $x\in M$.  We denote $dF:TM\rightarrow TN $'s components by $F^\a_i$. By Ricci identity, we have
\begin{equation}\label{heatdF}
    \frac{\partial}{\partial t} F^\a_i=\Delta F^\a_i-R^l_i F_l^\alpha+\tilde R^\alpha_{\e\gamma\delta}F^\e_i F_k^\gamma F_l^\delta g^{kl}-g^{kl}F^\a_l \nabla_i\nabla_k \phi -g^{kl} (F^\a_l)_{|i}\partial_k \phi 
\end{equation} 

Therefore,   $H_{ij}=h_{\a\b}F_i^\a F_j^\b$ satisfies 
\begin{equation}
\begin{split}
\frac{\partial }{\partial t} H_{ij}=&h_{\a\b}\left(\frac{\partial }{\partial t}  F_i^\a\right) F_j^\b+h_{\a\b}F_i^\a\left(\frac{\partial }{\partial t}   F_j^\b\right)\\
=&h_{\a\b}\left(\Delta F^\a_i-R^l_i F_l^\alpha+\tilde R^\alpha_{\e\gamma\delta}F^\e_i F_k^\gamma F_l^\delta g^{kl}-g^{kl}F^\a_l \nabla_i\nabla_k \phi -g^{kl} (F^\a_l)_{|i}\partial_k \phi \right) F_j^\b\\
&+h_{\a\b} F_i^\a\left(\Delta F^\b_j-R^l_j F_l^\b+\tilde R^\b_{\e\gamma\delta}F^\e_j F_k^\gamma F_l^\delta g^{kl}-g^{kl}F^\b_l \nabla_j\nabla_k \phi -g^{kl} (F^\b_l)_{|j}\partial_k \phi \right)\\
=&h_{\a\b}\left(\Delta F^\a_i\, \,F^\b_j+F^\a_i\, \Delta F^\b_j  \right)-(R^l_i+\nabla^l\nabla_i \phi ) H_{lj}-(R^l_j+\nabla^l\nabla_j \phi  )H_{li}\\
&-2g^{kl}\tilde R(F_i,F_k,F_l,F_j )-h_{\a\b}g^{kl}\partial_k\phi \left((F^\a_l)_{|i}F^\b_j +F^\a_i(F^\b_l)_{|j} \right) ,
\end{split}
\end{equation}
while
\begin{equation}
    \begin{split}
        \Delta H_{ij}=&g^{kl}\nabla_k \nabla_l (h_{\a\b}F_i^\a F_j^\b)\\
=&g^{kl}h_{\a\b}\nabla_k (\nabla_l F_i^\a\, F_j^\b +   F_i^\a \nabla_l F_j^\b)\\
=&h_{\a\b}\left(\Delta F^\a_i\,F^\b_j+F^\a_i\Delta F^\b_j  \right)+ 2g^{kl}h_{\a\b}\nabla_k F^\a_i \nabla_l F^\b_j.
    \end{split}
\end{equation}

Hence a subtraction gives us
    \begin{equation}
\begin{split}
\heat H_{ij}=-&\bigg[ (R^\phi)_i^l H_{lj}+(R^{\phi})_j^lH_{il}+
2g^{kl}\wt R(F_l, F_i, F_k, F_j)
\\
 & +2g^{pq} (F^\a_i)_{|p} (F^\b_j)_{|q} h_{\a\b}+g^{kl}\partial_k\phi \nabla_l H_{ij}   \bigg],
\end{split}
\end{equation}
where $\Delta$ is the Laplacian on $(M,g)$ and $(R^\phi)^i_j$ is a short hand of $R^i_j+\nabla^i\nabla_j\phi$. Since $g$ is a fixed metric on $M$, we immediately get (\ref{heatalpha}) for $\alpha$.

The second equation follows from taking trace of the first since $|d F|^2=\tr_g H=g^{ij} H_{ij}$.
\end{proof}

{\color{red}

}

At $x\in M$, we let $\lambda_i$ be the singular value of $dF$ so that the eigenvalues of $F^*h$ with respect to $g$ are $\{\lambda_i^2\}_{i=1}^n$ where $\lambda_i=0$ for $i>\ell=\min\{m,n\}$. 

\begin{prop}\label{prop:tensorMP}
Suppose $(M^n,g)$ and $(N^m,h)$ are two compact manifolds such that $\ell=\min\{n,m\}\geq 2$ and 
\begin{enumerate}
    \item[(a)] $\mathrm{sec}(h)\geq 0$;
    \item[(b)] there exists $\phi\in C^\infty(M)$ such that  for all unit vector $u\in TM$ and $v\in TN$, 
    \begin{equation}
        \Ric^\phi_g(u,u)\geq \Ric_h(v,v).
    \end{equation}
\end{enumerate}
If $F:M\times [0,T]\to N$ is a solution to the $\phi$-harmonic map heat flow such that the initially $(1-\e)^2g-f^*h$ is $2$-nonnegative for some $1>\e\geq 0$, then $(1-\e)^2g-F^*h$ remains $2$-nonnegative for all $t\in [0,T]$. 
\end{prop}
\begin{proof}
   Suppose $\e>0$, we define $\tilde h=(1-\e)^{-2}h$ 
   so that $ F$ is still a $\phi$-harmonic map heat flow with respect to $g$ and $\tilde h$.  If $u\in TM$ and $v\in TN$ are such that $g(u,u)=1=\tilde h(v,v)$.  Then $\tilde v=(1-\e)^{-1}v$ satisfies $h(\tilde v,\tilde v)=1$ and hence
\begin{equation}
      \Ric_g^\phi ( u, u)\geq \Ric_{ h}(\tilde v,\tilde v)\geq \Ric_{\tilde h}(v,v)\geq 0.
\end{equation}
Hence,  it is sufficient to consider the case of $\e=0$. 

By tensor maximum principle (for instances see \cite[Theorem A.21]{ChowBook}), it suffices to verify the null vector condition. That is to show that if $\a$ is $2$-nonegative and $\a(e_1,e_1)+\a(e_2,e_2)=0$ at $(x_0,t_0)$, then $\heat (\a_{11}+\a_{22})\geq 0$ at $(x_0,t_0)$.  At $(x_0,t_0)$, we let $\{\lambda_i\}_{i=1}^n$ be the singular values of $dF$ so that $H_{ij}$ is diagonalized with eigenvalues $\{\lambda_i^2\}_{i=1}^n$ with order $\lambda_1^2\geq \lambda_2^2\geq ...\geq \lambda_n^2$. Then the eigenvalues of $\a$ with respect to the $2$-nonnegative cone is realized by $e_1$ and $e_2$, i.e. $\a_{11}+\a_{22}=2-\lambda_1^2-\lambda_2^2$. 

By (\ref{heatalpha}), we can compute the evolution of $\a_{11}+\a_{22}$, evaluating at $(x_0,t_0)$ 
\begin{equation}\label{null-vec}
\begin{split}
\heat (\a_{11}+\a_{22})\Big|_{(x_0,t_0)}
=&2\bigg[ (R^\phi)_1^lH_{l1}-
g^{kl} (F^*\tilde R)_{k11l}+g^{kl}F^\a_{1k} F_{1l}^\b h_{\a\b}\bigg]\\
&+2\bigg[ (R^\phi)_2^lH_{l2}-
g^{kl} (F^*\tilde R)_{k22l}+g^{kl}F^\a_{2
k} F_{2l}^\b h_{\a\b}\bigg]\\
\geq& 2\sum_{i=1}^2 (R^\phi)_{ii}\lambda_i^2-2\sum_{i=1}^2 \sum_{k=1}^n \tilde R_{ikki}\lambda_i^2\lambda_k^2\\
\geq& 2\sum_{i=1}^2 \tilde R_{ii}\lambda_i^2-2\sum_{i=1}^2 \sum_{k=1}^n \tilde R_{ikki}\lambda_i^2\lambda_k^2\\
=& 2\sum_{i=1}^2 \lambda_i^2 \left( \sum_{k=1}^m  \tilde R_{ikki} - \sum_{k=1}^\ell  \lambda_k^2\tilde R_{ikki} \right)
\end{split}
\end{equation}
where we have used $\lambda_i=0$ for $i>\ell$, $R_{ii}^\phi\geq \tilde R_{ii}$ and $\nabla_X (\a_{11}+\a_{22})=0$ at $(x_0,t_0)$ for all $X\in T_{x_0}M$.

We now estimate the right hand side from below. Since $\mathrm{sec}(h)\geq 0$ and $\ell\leq m$,
\begin{equation}
    \begin{split}
  \sum_{k=1}^m  \tilde R_{ikki} - \sum_{k=1}^\ell  \lambda_k^2\tilde R_{ikki}
       &\geq  \sum_{k=1}^\ell  \tilde R_{ikki} - \sum_{k=1}^\ell  \lambda_k^2\tilde R_{ikki}\\
       &=\sum_{k=1}^\ell (1-\lambda_k^2) \tilde R_{ikki}.
    \end{split}
\end{equation}

Using also $|\lambda_i|\leq 1$ for each $i>1$,


\begin{equation}
    \begin{split}
        \sum_{i=1}^2 \lambda_i^2 \left( \sum_{k=1}^\ell  (1- \lambda_k^2)\tilde R_{ikki}\right)&\geq  \lambda_1^2  (1- \lambda_2^2)\tilde R_{1221} +\lambda_2^2 (1- \lambda_1^2)\tilde R_{1221}\\
        &= \frac12 \left( \lambda_1^2-\lambda_2^2\right)^2\tilde R_{1221} \geq 0
    \end{split}
\end{equation}
where we have used $\lambda_1^2+\lambda_2^2=2$. Combining this with \eqref{null-vec}, this verifies the null-vector condition. This completes the proof.
\end{proof}

We now consider the energy density along the flow under the curvature assumptions. Under a slightly weaker contractive property, we can prove a monotonicity on the total energy density. 
\begin{prop}\label{prop:C1-after2non}
Suppose $(M^n,g)$ and $(N^m,h)$ are compact manifolds satifying the assumptions in Proposition~\ref{prop:tensorMP}. If $F:M\times [0,T]\to N$ is a solution to the $\phi$-harmonic map heat flow which remains area non-increasing, i.e. $||\Lambda^2 dF||_\infty\leq 1$, for all $t\in [0,T]$, then on $M\times [0,T]$,
\begin{equation}\label{heatdF^2}
    \heat |dF|^2 \leq -2|\nabla d F|^2+\langle \nabla \phi,\nabla |dF|^2 \rangle.
\end{equation}
\end{prop}
\begin{proof}
Our assumption in term of the singular values says that  $|\lambda_i\lambda_j|\leq 1$ for all $i\neq j$.  Using \eqref{heatgradF}, $\mathrm{sec}(h)\geq 0$ and the area non-increasing property, we compute 
\begin{equation}\label{ineq-Bochner-1}
\begin{split}
    \heat |d F|^2
    =&-2|\nabla d F|^2 -2 \sum_{i=1}^n (R^\phi)_{ii}\lambda_i^2+2\sum_{i,k=1}^n \tilde R_{ikki}\lambda_i^2\lambda_k^2+\langle \nabla \phi,\nabla |dF|^2 \rangle \\
    \leq &-2|\nabla d F|^2  -2\sum_{i=1}^\ell \tilde R_{ii}\lambda_i^2+2\sum_{i=1}^\ell\sum_{k=1}^\ell \tilde R_{ikki}|\lambda_i\lambda_k|+\langle \nabla \phi,\nabla |dF|^2 \rangle\\
    \leq& -2|\nabla d F|^2  -2\sum_{i=1}^\ell \tilde R_{ii}\lambda_i^2+\sum_{i=1}^\ell\sum_{k=1}^\ell (\lambda_i^2+\lambda_k^2)\tilde R_{ikki}+\langle \nabla \phi,\nabla |dF|^2 \rangle\\
    =&-2|\nabla d F|^2  -2\sum_{i=1}^\ell \tilde R_{ii}\lambda_i^2+2\sum_{i=1}^\ell\sum_{k=1}^\ell \lambda_i^2\tilde R_{ikki}+\langle \nabla \phi,\nabla |dF|^2 \rangle.
\end{split}
\end{equation}

Since $\mathrm{sec}(h)\geq 0$ and $m\geq \ell$,  the third term can be estimated as 
\begin{equation}\label{ineq-Bochner-2}
 \begin{split}
    \sum_{i=1}^\ell   \sum_{k=1}^\ell  \lambda_i^2 \tilde R_{ikki}&\leq     \sum_{i=1}^\ell    \sum_{k=1}^m  \lambda_i^2 \tilde R_{ikki}=  \sum_{i=1}^\ell  \lambda_i^2 \tilde R_{ii}.
    \end{split}
\end{equation}

From this, the sum of the second and the third term is non-positive, this completes the proof. 
\end{proof}

We end this section by recalling some well-known results of the harmonic map heat flow.

\begin{thm}\label{thm:standard-existence}
Suppose $(M,g)$ and $(N,h)$ are two smooth compact manifolds and $\phi\in C^\infty(M)$. If $f:M\to N$ is a smooth map between $M$ and $N$, then there exists a short-time solution to the $\phi$-harmonic map heat flow $F:M\times [0,T]\to N$.
\end{thm}
\begin{proof} 
When $\phi$ is a constant, it follows from the celebrated work of Eells-Sampson \cite{EellsSampson1964}. Since $\phi$ is smoothly controlled and $M,N$ are both compact, the contraction mapping argument can be carried over directly. We refer readers to \cite[Chapter 5]{LiWangBook}. Alternatively, we can obtain a solution $F$ by compositing a standard harmonic map heat flow $\bar F$ with $\Phi_t:M\to M$ where \begin{equation}
    \left\{
    \begin{array}{ll}
         \partial_t \Phi_t=-\nabla \phi;  \\
         \Phi_0=\mathrm{Id}. 
    \end{array}
    \right.
\end{equation}
\end{proof}

The following regularity theory of harmonic map heat flow is standard. 
\begin{prop}\label{prop:standard-boots}
Suppose the $\phi$-harmonic map heat flow $F:M^n\times [0,T]\to N^m$ satisfies $$|dF|\leq \Lambda$$ on $M\times [0,T]$, then for all $k\in \mathbb{N}$, there exists $C(n,m,k,g,h,\phi),S(n,m,k,g,h,\phi)>0$ such that 
$$\sup_M |\nabla^k dF|^2\leq \frac{C(n,m,k,g,h,\phi)}{t^k}$$
on $M\times (0,T\wedge S]$. 
\end{prop}
\begin{proof}
Since $(M,g)$ and $(N,h)$ are compact manifolds, it follows from standard parabolic theory. For a maximum principle proof, we refer readers to the proof of \cite[Theorem 3.1]{LeeTam2022} in case when $\phi$ is constant. Since $\phi$ is smoothly bounded, the proof can be modified accordingly thanks to the assumed gradient estimate. 
\end{proof}

\section{Longtime convergence and rigidity}


In this section, we will finish the proof of Theorem~\ref{intro-main} by studying the long-time behaviour of the harmonic map heat flow \eqref{eqn:HMHF}.

\begin{proof}[Proof of Theorem~\ref{intro-main}]
We assume that $f$ is not homotopically trivial.  By Theorem~\ref{thm:standard-existence}, it admits a short-time solution to \eqref{eqn:HMHF} with initial data $f:M^n\to N^m$. We let $T_{max}$ be the maximal existence time. 

We first claim that $T_{max}=+\infty$. It follows from Proposition~\ref{prop:tensorMP} that $g-F^*h$ remains $2$-nonnegative for all $t\in [0,T_{max})$. Hence, Proposition~\ref{prop:C1-after2non} implies that the energy density $|d F|^2$ satisfies
\begin{equation}
     |d F|^2 \leq \sup_M |d f|^2=C_0(n,m,f,g,h,\phi)
\end{equation}
on $M\times [0,T_{max})$ by maximum principle. In particular, Proposition~\ref{prop:standard-boots} implies that $F$ is bounded uniformly in $C^k_{loc}$ for all $k\in \mathbb{N}$ as $t\to T_{max}$. It follows that $T_{max}$ must be $+\infty$. Moreover, integrating the inequality (\ref{heatdF^2}) in Proposition~\ref{prop:C1-after2non} implies that for all $t>0$,
\begin{equation}\label{energy-mono}
\begin{split}
   &\quad  \int_M |d F(t)|^2 e^{\phi} \,d\mathrm{vol}_g\\
   &=\int^t_0 \frac{d}{ds}\left( \int_M |d F|^2 e^{\phi} \,d\mathrm{vol}_g \right)ds+ \int_M |d f|^2 e^{\phi}\,d\mathrm{vol}_g\\
    &\leq \int^t_0 \int_M\left(   \Delta |d F|^2-2|\nabla dF|^2 +\langle \nabla \phi,\nabla |dF|^2\rangle \right)e^{\phi} \,d\mathrm{vol}_g ds+ \int_M |d f|^2 e^{\phi}\,d\mathrm{vol}_g\\
    &=  \int^t_0 \int_M -2|\nabla dF|^2 e^{\phi} \,d\mathrm{vol}_g ds + \int_M |d f|^2 e^{\phi}\,d\mathrm{vol}_g.
\end{split}
\end{equation}

Together with the higher order regularity, it follows that there exists $t_k\to +\infty$ such that $F(t_k)$ sub-converges to $F_\infty:M\to N$ smoothly such that $\nabla d F_\infty\equiv 0$ on $M$. Moreover,  $g-F_\infty^*h$ remains $2$-nonnegative and hence $F_\infty$ is also area non-increasing. Thanks to the smooth convergence, $F_\infty$ is homotopic to the initial data $f$ and hence not homotopically trivial. 

We now claim that differential of $F_\infty$ has largest $\ell$ singular values being $1$. We will omit the sub-script. 
Now applying the Bochner formula in Lemma~\ref{lma:evo-F} for the time independent map $F=F_\infty$, we see that 
\begin{equation} \label{eqn:rigidity-harmonic}
\begin{split}
   0&= \int_M\Delta |d F|^2  \,d\mathrm{vol}_g\\
   &= \int_M 2|\nabla d F|^2 +2R^{ij}F_i^\a F_j^\b h_{\a\b} -2g^{kl}g^{ij} (F^*\mathrm{Rm}_h)_{iklj}
   \,d\mathrm{vol}_g\\
   &= \int_M 2(R^\phi)^{ij}F_i^\a F_j^\b h_{\a\b} -2g^{kl}g^{ij} (F^*\mathrm{Rm}_h)_{iklj}
   \,d\mathrm{vol}_g-2\int_M   \nabla^i \nabla^j \phi \cdot H_{ij} \,d\mathrm{vol}_g\\
   &= \int_M 2(R^\phi)^{ij}F_i^\a F_j^\b h_{\a\b} -2g^{kl}g^{ij} (F^*\mathrm{Rm}_h)_{iklj}
   \,d\mathrm{vol}_g.
   \end{split}
\end{equation}
where we have used $\nabla dF=0$ and Stoke's Theorem on the last equality.

We now proceed as in the proof of Proposition~\ref{prop:C1-after2non}. We need to trace carefully the inequality used. At $x\in M$, we let $\lambda_i$ be the singular value of $dF_\infty$ so that the eigenvalues of $F^*h$ with respect to $g$ are $\{\lambda_i^2\}_{i=1}^n$. We assume $\lambda_1^2\geq \lambda_2^2\geq...\geq \lambda_n^2$. We now extract geometric information from the equality in Bochner formula. 

We now examine each equality. Recall from the proof of Proposition~\ref{prop:C1-after2non}, all inequalities in \eqref{ineq-Bochner-1} and \eqref{ineq-Bochner-2} become equalities at each $x\in M$. Since $\mathrm{sec}(h)>0$ on $M$, we must have 
\begin{enumerate}
    \item[(i)]$\Ric^\phi_{ii}=\widetilde \Ric_{ii}$ for all $1\leq i\leq \ell$ where $\lambda_i\neq 0$;
    \item[(ii)] $|\lambda_i\lambda_j|$ is either $1$ or $0$, for all $1\leq i\neq j\leq \ell$;
    \item[(iii)] $\lambda_i=\lambda_j$ for all $1\leq i,j\leq \ell$;
    \item[(iv)] either
    \begin{enumerate}
        \item $m=\ell$ or;
        \item $\lambda_i\equiv 0$ for all $1\leq i\leq \ell $.
    \end{enumerate}
\end{enumerate}

Since $f$ (and hence $F$) is assumed to be homotopically non-trivial, the case (b) in (iv)  cannot happen since $\nabla dF=0$. Hence we must have $m\leq n$.  Using again $\nabla dF=0$, it suffices to consider one point. If for some $i\neq j$, $\lambda_i\lambda_j=0$, then we must have $\lambda_k\equiv 0$ for all $1\leq k\leq \ell$ by (iii) which is also impossible. Hence, we must have $\lambda_i^2\equiv 1$ for all $1\leq i\leq \ell=m$. In particular, $|dF|^2=\tr_g F^*h=\sum_{i=1}^\ell \lambda_i^2=\ell=m$ on $M$.


We now show that differential of $f$ also has largest $\ell$ singular values being $1$. By passing \eqref{energy-mono} to $F=F_\infty $, we see that 
\begin{equation}
    \begin{split}
      m \mathrm{Vol}_\phi(M,g)  =\int_M |d F_\infty|^2e^{\phi} d\mathrm{vol}_g \leq \int_M |d f|^2 e^{\phi} d\mathrm{vol}_g.
    \end{split}
\end{equation}

In term of singular value of $f$, the energy density on the right hand side can be re-written as 
\begin{equation}
|d f|^2=\frac12 \left( \sum_{i=1}^m \lambda_i^2 +\sum_{j=1}^m\lambda_j^2  \right) \leq m 
\end{equation}
where we have used the $2$-nonegative of $g-f^*h$. Therefore, we conclude that 
\begin{equation}
    \int_M  |d f|^2 e^{\phi}\,d\mathrm{vol}_g \leq m\mathrm{Vol}_\phi(M,g) .
\end{equation}
and hence the singular value of $df$ satisfies $\lambda_i^2=1$ for all $i=1,...,\ell=m$ everywhere. Thus $f$ is a Riemannian submersion (if $n>m$) or local isometry (if $n=m$)

 When $n=m$ and $f: (M,g)\rightarrow (N,h)$ is a local isometry, assumption (b) implies that 
\begin{equation}
    \inf\{ \Ric^\phi_g(u,u): u\in TM\}\geq \sup\{ \Ric_g(v,v): v\in TM\}
\end{equation}
and hence $\phi$ is sub-harmonic. Strong maximum principle implies that $\phi$ is constant on each connected component of $M$. Moreover, $M$ and $N$ must be Einstein manifolds since the Ricci curvature is constant. This completes the proof. 
\end{proof}

As a consequence of the method employed, we can relax the curvature conditions if $g-f^*h$ is strictly $2$-positive.
\begin{proof}[Proof of Theorem~\ref{intro-thm2}]
The proof is similar to that of Theorem~\ref{intro-main}. We only point out the necessary modifications. As in the proof of Theorem~\ref{intro-main}, we obtain a limiting smooth map $F=F_\infty:M\to N$ which is homotopic to $f$ and has $\nabla dF=0$. Now the stronger initial assumption implies that $g-F^* h$ is $2$-positive using Proposition~\ref{prop:tensorMP} with $\e>0$. This in particular implies $F$ is strictly area decreasing. 

We now claim that $F_\infty$ is homotopically trivial.  
It suffices to improve the curvature terms appeared in \eqref{ineq-Bochner-1} and \eqref{ineq-Bochner-2}, and thus in \eqref{eqn:rigidity-harmonic}. Denote $\lambda_i$ to be the singular value of $dF_\infty$ with respect to $g$ and $h$. Since $\frac12 (\lambda_i^2+\lambda_j^2)\leq 1-\e$ for some $\e>0$, $|\lambda_i\lambda_j|\leq 1-\e$ for all $i\neq j$ and hence, 
\begin{equation}
    \begin{split}
        \sum_{i,k=1}^n 2\lambda_i^2 \lambda^2_k \tilde R_{ikki}&\leq  \sum_{i,k=1}^\ell  2(1-\e)|\lambda_i \lambda_k| \tilde R_{ikki}\\
        &\leq (1-\e)\sum_{i,k=1}^\ell (\lambda_i^2+\lambda_k^2) \tilde R_{ikki}\\
        &= 2(1-\e) \sum_{i=1}^\ell  \lambda_i^2 \left(\sum_{k=1}^\ell  \tilde R_{ikki}\right)\\
        &\leq 2(1-\e) \sum_{i=1}^\ell  \lambda_i^2 \tilde R_{ii}. 
    \end{split}
\end{equation}

In particular, integrating the energy density yields 
\begin{equation}
    0\geq 2\e \int_M (R^\phi)^{ij}H_{ij}\,d\mathrm{vol}_g\geq 0
\end{equation}
If $\Ric^\phi$ is positive at $x_0\in M$, then $H(x_0)=0$ and hence $H\equiv 0$ on $M$ using $\nabla dF_\infty=0$. This shows that $f$ is homotopically trivial. If $\phi$ is a constant, then we can compute using contracted Bianchi to deduce
\begin{equation}
\begin{split}
0=\int_M R^{ij}H_{ij} \,d\mathrm{vol}_g=\frac12 \int_M \mathrm{scal}(g)\cdot \tr_gH \,d\mathrm{vol}_g\geq 0
\end{split}
\end{equation} 
so that $\mathrm{scal}(g)\cdot  \tr_gH \equiv 0$ on $M$. Suppose the scalar curvature of $g$ is positive somewhere, we have $H(x_0)=0$ for some $x_0\in M$. Using $\nabla d F_\infty=0$, $H$ is parallel and hence $H\equiv 0$ on $M$. This shows that $F_\infty$ is homotopically trivial and so does $f$. 
\end{proof}

It is worth mentioning that the proof is based on deforming a given map $f$ to a harmonic one using $\phi$-harmonic map heat flow. The stronger contractive property, i.e. $g-f^*h$ is 2-nonnegative is used in order to rule out finite time singularity. When $f$ is $\phi$-harmonic, i.e. $\Delta f=\langle df,d\phi\rangle $ on $M$. Thanks to the weaker assumptions in Proposition~\ref{prop:standard-boots}, we can conclude the following immediately following the same argument in that of Theorem~\ref{intro-main}.

\begin{cor}
Let $(M^n,g)$ and $(N^m,h)$ be two compact manifolds satisfying the assumptions in Theorem~\ref{intro-main}. If $f:M\to N$ is an area non-increasing $\phi$-harmonic map, then the following holds: 
\begin{enumerate}
\item[(I)] if $f$ is homotopy trivial, then $f$ is a constant map; 
\item[(II)] if $m>n$, then $f$ is  homotopy trivial; 
    \item[(III)] if $n>m$ and $f$ is not  homotopy trivial, then $f$ is a Riemanian submersion;
    \item[(IV)] if $m=n$ and $f$ is not  homotopy trivial, then $f$ is a local isometry and $\phi$ is a constant on each connected component of $M$. Moreover, $M,N$ are Einstein manifolds with positive sectional curvature.
    \end{enumerate}
\end{cor}
\begin{proof}
Since $f$ is $\phi$-harmonic, then we might take $F(t)\equiv f$ as a immortal solution to the $\phi$-harmonic map heat flow. Following the argument in Theorem~\ref{intro-main}, we conclude that if $f$ (and hence $F$) is not homotopy trivial, then we must have (II), (III) and (IV). It remains to show the rigidity of homotopy trivial map. It follows from the rigidity of equality in \eqref{ineq-Bochner-1} and \eqref{ineq-Bochner-2} that if we must have $\lambda_i\equiv 0$ for all $i$ and thus $f$ is constant. 
\end{proof}

Similarly,  we have the following:
\begin{cor}
    Let $(M^n,g)$ and $(N^m,h)$ be two compact manifolds satisfying the assumptions in Theorem~\ref{intro-thm2}. If $f:M\to N$ is a  strictly area decreasing $\phi$-harmonic map, then $f$ is constant if one of the following holds:
\begin{enumerate}
\item[(a)] $\Ric_g^\phi$ is positive somewhere;
\item[(b)] $\phi$ is a constant on each connected component of $M$ and $\mathrm{scal}(g)$ is positive somewhere.
\end{enumerate}
\end{cor}
\begin{proof}
It is easy to see from the proof of Theorem~\ref{intro-thm2} that if the limiting $\phi$-harmonic map is area decreasing, then $f^*h=H\equiv 0$ on $M$. This shows that $f$ is constant.  
\end{proof}

\end{document}